\documentclass{article}

\usepackage{arxiv}

\usepackage[utf8]{inputenc} 
\usepackage[T1]{fontenc}    
\usepackage[hidelinks]{hyperref}       
\usepackage{url}            
\usepackage{booktabs}       
\usepackage{amsfonts, amsmath}       
\usepackage{amsthm}
\usepackage{nicefrac}       
\usepackage{microtype}      
\usepackage{natbib}
\usepackage{doi}
\usepackage{fancyhdr, graphicx, float, geometry}
\usepackage{enumitem}
\usepackage{tikz}
\usepackage{subcaption}

\newtheorem{remark}{Remark}
\newtheorem{theorem}{Theorem}
\newtheorem{lemma}{Lemma}
\newtheorem{corollary}{Corollary}

\newcommand{\eps}{\varepsilon}
\newcommand{\rr}{\mathbb{R}}
\newcommand{\Hell}{\mathrm{Hell}}
\newcommand{\chem}{\mathrm{chem}}
\newcommand{\KS}{\mathrm{KS}}
\newcommand{\KL}{\mathrm{KL}}

\title{Multiscale convergence of the inverse problem for chemotaxis in the Bayesian setting}

\author{ \href{https://orcid.org/0000-0002-5090-9218}{\includegraphics[scale=0.06]{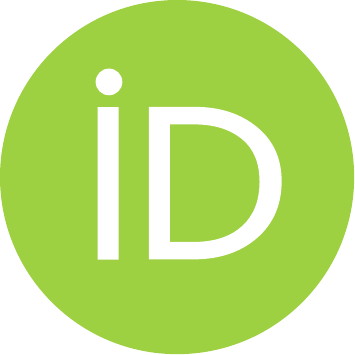}\hspace{1mm}Kathrin Hellmuth}\\
	Department of Mathematics\\
	University of W\"urzburg\\
	97074 Würzburg, Germany\\
	\texttt{kathrin.hellmuth@mathematik.uni-wuerzburg.de}\\
	\And 
	\href{https://orcid.org/0000-0003-2033-8204}{\includegraphics[scale=0.06]{orcid.pdf}\hspace{1mm}Christian Klingenberg} \\
	Department of Mathematics\\
	University of W\"urzburg\\
	97074 Würzburg, Germany\\
	\texttt{klingen@mathematik.uni-wuerzburg.de}\\
	\And 
	\href{https://orcid.org/0000-0001-9210-8948}{\includegraphics[scale=0.06]{orcid.pdf}\hspace{1mm}Qin Li} \\
	Department of Mathematics\\
	University of Wisconsin-Madison\\
	Madison, WI, 53705, USA\\
	\texttt{ qinli@math.wisc.edu}\\
	 \And 
	 \hspace{1mm}Min Tang \\
	 School of Mathematics\\
	 Shanghai Jiaotong University\\
	 Shanghai, 200240, China\\
	 \texttt{ tangmin@sjtu.edu.cn}
}

\begin{document}
\maketitle
 \begin{abstract}
 	Chemotaxis describes the movement of an organism, such as single or multi-cellular organisms and bacteria, in response to a chemical stimulus. Two widely used models to describe the phenomenon are the celebrated Keller-Segel equation and a chemotaxis kinetic equation. These two equations describe the organism movement at the macro- and mesoscopic level respectively, and are asymptotically equivalent in the parabolic regime. How the organism responds to a chemical stimulus is embedded in the diffusion/advection coefficients of the Keller-Segel equation or the turning kernel of the chemotaxis kinetic equation. Experiments are conducted to measure the time dynamics of the organisms' population level movement when reacting to certain stimulation. From this one infers the chemotaxis response, which constitutes an inverse problem. \\ In this paper we discuss the relation between both the macro- and mesoscopic inverse problems, each of which is associated to two different forward models. The discussion is presented in the Bayesian framework, where the posterior distribution of the turning kernel of the organism population is sought after. We prove the asymptotic equivalence of the two posterior distributions.
 \end{abstract}
\section{Introduction}
Chemotaxis is the phenomenon of organisms directing their movements upon certain chemical stimulation. Every motile organism exhibits some type of chemotaxis. Mathematically, there are two main-stream mathematical models used to describe this phenomenon: One at the macroscopic population level and the other at the mesoscopic level.

The most famous model in the first category is the Keller-Segel equation, introduced in~\cite{patlak1953random,KSmodel,KellerSegelTravellingBands}. The equation traces the evolution of bacteria density when chemical stimulation is introduced to the system:
\begin{align}\label{eq:KS}
	&\frac{\partial }{\partial t} \rho - \nabla\cdot (D\cdot \nabla \rho )+ \nabla \cdot  (\rho \Gamma) =0,
\end{align}
where $\rho(x,t)$ is the cell density at location $x$ at time $t>0$. In this equation, both the advection term and the diffusion process integrate the external chemical density information, meaning both the diffusion matrix $D[c] (x,t)$ and the drift vector $\Gamma[c]( x,t)$ depend on the chemoattractant's density function $c$, which serves as a meta parameter  determining the $(x,t)$ dependence.

However, the model is inaccurate in certain regimes. It overlooks the detailed bacteria's reaction to the chemoattractants, and is thus macroscopic in nature. This inspires the second category of modeling, where the motion of  individual bacteria is accounted. The associated modeling is thus mesoscopic. When bacterial movements are composed of two states: running in a straight line with a given velocity $v$ and tumbling from one velocity  $v$ to another $v'$, the according mathematical model is termed the run-and-tumble model. It  is described by the mesoscopic chemotaxis equation~\cite{perthame2006transport,chalub2004kinetic,AltChemotaxis}:
\begin{align}\label{eq:unscalechemotaxis}
	\frac{\partial }{\partial t} f(x,t,v) + v\cdot &\nabla_x f(x,t,v) = \mathcal{K}[c](f) \\
	&:= \int_V K[c](x,t,v,v') f(x,t,v') - K[c](x,t,v',v) f(x,t,v) dv'.\nonumber
\end{align}
In the equation, $f(x,t,v)$ is the population density of bacteria with velocity $v\in V\subset\rr^3$ at space point $x\in \rr^3$ at time $t>0$. The tumbling kernel $ K[c](x,t,v,v' )$  encodes the probability of bacteria changing from velocity $v'$ to $v$. The shape of this probability  depends on the chemoattractant's density function $c$.

Abbreviating the notation and calling $f':= f(x,t,v')$ and $K'[c]:= K[c](x,t,v',v)$ as in \cite{chalub2004kinetic}, the tumbling term on the right-hand side of equation~\eqref{eq:unscalechemotaxis} reads
$$
\mathcal{K}[c](f) = \int_V K[c] f' - K'[c] f dv'\,.
$$
Because bacteria are usually assumed to move 
with constant speed, conventionally we have $V=\mathbb{S}^{n-1}$. Moreover, since the cell doubling time is much longer than the chemotaxis time scale, we remove the birth-death effect from the equation.



Both models above are empirical in nature. The coefficients, such as $D$, $\Gamma$ and $K$ that encode the way bacteria respond to the environment are typically unknown ahead of time. Since the chemoattractant concentration $c$ depends on space and time, so do $D$, $\Gamma$ and $K$. However, except for very few well studied bacteria, these quantities are not explicitly known and cannot be measured directly. One thus needs to design experiments and use measurable quantities to infer the information. This constitutes the inverse problem we study. One such experiment was reported in~\cite{giometto2015generalized} where the authors studied phototaxis and use video recording of the seaweed motion ($\rho$ in time) to infer $D$ and $\Gamma$ in~\eqref{eq:KS}.

There are various ways to conduct inverse problems, and in this paper, we take the viewpoint of Bayesian inference. This is to assume that the coefficients are not uniquely configured in reality, but rather follow a certain probability distribution. The measurements are taken to infer this probability. In the process of such inference, one nevertheless needs to incorporate the forward model. The two different forward models described above then lead to two distinctive posterior distributions as the inference.

One natural question is to understand the relation between the two resulting posterior distributions. In this article, we answer this question by asymptotic analysis. To be specific, we will show that the two models are asymptotically equivalent in the long time and large space regime, and $(D,\Gamma)$ can be uniquely determined by a given $K$. As such, the associated two inverse problems are asymptotically equivalent too. 
The equivalence is characterized by the distance (we use both the Kullback–Leibler divergence and the Hellinger distance) between the two corresponding posterior distributions. We show that this distance vanishes asymptotically as the Knudsen number, a quantity that measures the mean free path between two subsequent tumbles, becomes arbitrarily small.

The rest of the paper is organized as follows: In section \ref{sec:2setup} we present the asymptotic relation between the two forward models. This can be seen as an adaption of the results in \cite{chalub2004kinetic} to our setting. The analysis serves as the foundation to link the two inverse problems. In section~\ref{sec:3setup} we formulate the Bayesian inverse problems corresponding to the scaled chemotaxis equation and the Keller Segel model as underlying models. The well-posedness and convergence of the two corresponding posterior distributions is shown in section~\ref{sec:4PosteriorConv}. The results are summarized and discussed in section \ref{sec:5Summary}.

We should stress that both mathematical modeling of chemotaxis and Bayesian inference are active research areas. In formulating our problems, we select the most widely-accepted models and methods.

For modeling chemotaxis, the two models~\eqref{eq:KS}-\eqref{eq:unscalechemotaxis} are the classical ones, and were derived from the study of a biased random walk~\cite{AltChemotaxis,patlak1953random}. They assume the organisms passively depend on the environment. When bacteria actively respond and change the environment, a parabolic or elliptic equation for $c$ can be added to describe such feedback to the environment~\cite{KSmodel,KellerSegelSlimeMold,KellerSegelTravellingBands}. The coupled system consisting of equation \eqref{eq:KS} and a parabolic equation for $c$, where the chemo-attractant is assumed to be produced by the bacteria population, can exhibit blow-up solutions. Therefore, some particular form of $D[c]$, $\Gamma[c]$ are proposed to eliminate the unwanted behavior. These models include volume filling~\cite{volumefilling}, quorum sensing models~\cite{quorumsensing}, or the flux limited Keller Segel system~\cite{Perthame2018FluxLimitedKS}. On the kinetic level, additional variables were introduced to describe the intracellular responses of the bacteria to the chemoattractant in the signalling pathway~\cite{Othmer2004InternalVar,Si2012InternalVar,MinTang2014InternalVar,Tang2016PathwayBasedMacLimits}. Asymptotic limits of these newer models sometimes reveal interesting phenomenon such as fractional diffusion~\cite{Perthame2017FractionalDiffusion}. The asymptotic equivalence of the classical model to the Keller Segel model was extensively studied e.g. in~\cite{AltChemotaxis,OthmerHillen2Chemo,OthmerAltDispersalModels,chalub2004kinetic}. In particular, the current paper heavily depends on the techniques shown in~\cite{chalub2004kinetic}.

There is also a vast literature on inverse problems. For Bayesian inference perspective in scientific computing, interested readers are referred to monographs~\cite{stuart2010inverse,Stuart2015BayesianInversionHandbookofUQ} and the references therein. In comparison, linking two or multiple inverse problems in different regimes are relatively rare. In \cite{newton2020diffusive}, the authors studied the asymptotic equivalence between the inverse kinetic radiative transport equations and its macroscopic counterpart, the diffusion equation. In  \cite{Abdulle2018BayesianConvParametPDE}, the convergence of Bayesian posterior measures for a parametrized elliptic PDE forward model was shown in a similar fashion.

\section{Asymptotic analysis for kinetic chemotaxis equations and the Keller-Segel model}\label{sec:2setup}
The two problems we will be using are chemotaxis kinetic equation and the Keller-Segel equation. We review these two models in this section and study their relation. It serves as a cornerstone for building the connection of the two associated inverse problems.

Throughout the paper, we assume the chemoattractant density $c$ is one given and fixed function of $(x,t)$ and is \emph{not} produced or consumed by the bacteria. While this is an approximation, it is valid in many experiments where one has tight control over the matrix environment. We drop the dependence of $K, D,\Gamma$ on the fixed $c$ in the notation.


We claim, and will show below that the two equations~\eqref{eq:unscalechemotaxis} and~\eqref{eq:KS} are asymptotically equivalent in the long time large space regime. Denote $\eps$ the scaling parameter, then in a parabolic scaling, the chemotaxis equation to be considered  has the following form:
\begin{equation}\label{eq:chemotaxis}
    \begin{aligned}
	{\eps^2}\frac{\partial }{\partial t} f_{ {\eps}}(x,t,v) +  {\eps} v\cdot &\nabla_x f_{ {\eps}}(x,t,v) = \mathcal{K}_{ {\eps}}(f_{ {\eps}}) \\
	&:= \int_V K_{ {\eps}}(x,t,v,v') f_{ {\eps}}(x,t,v') - K_{ {\eps}}(x,t,v',v) f_{ {\eps}}(x,t,v) dv'\\
	f_{\eps}(x,0,v)&= f_0(x,v).
    \end{aligned}
\end{equation}
Formally, when $\eps\to0$, the tumbling term dominates the equation and we expect, in the leading order:
\[
f_\eps\to f_\ast\,,\quad\text{with}\quad \mathcal{K}_\ast(f_\ast)=0\,,
\]
where $\mathcal{K}_\ast$ can be viewed as the limiting operator as $\mathcal{K}_\eps$. This means the limiting solution is almost in the null space of the limiting tumbling operator. Furthermore, due to the specific form of the tumbling operator, one can show that under certain conditions such null space is one dimensional, compare e.g. \cite{chalub2004kinetic} Lemma 2 and following derivations. We thus  formally write
\[
\mathcal{N}(\mathcal{K}_\ast) = \{\alpha F: \alpha\in\mathbb{R}\,,\text{with}\quad\int_VF d{v} = 1\}\,,
\]
and denote $f_\ast = \rho F$. Conventionally we call $F$ the local equilibrium. Due to the form of $\mathcal{K}$, this is a function only of $v$. Inserting this formula back into~\eqref{eq:chemotaxis} and perform asymptotic expansion up to the second order, and following~\cite{chalub2004kinetic}, we find that $\rho$ satisfies the Keller-Segel equation:
\begin{align}\label{eq:KellerSegel}
	\frac{\partial }{\partial t} \rho - &\nabla\cdot (D\cdot \nabla \rho )+ \nabla \cdot  (\rho \Gamma) =0,\\
	\rho(x,0)&=\rho_0(x)= \int_V f_0(x,v)\, dv.\nonumber
\end{align} 
A rigorous proof of the convergence of a subsequence of $f_{\eps}$ can be found in \cite{chalub2004kinetic}, theorem 3, where the authors discussed a nonlinear extension of the present model. \\
\\
From now on, we confine ourselves to kernels having the form of
\begin{equation}\label{eqn:K_form}
    K_{\eps} = K_0+\eps K_1\,.
\end{equation}
\begin{remark}
Because our aim is to compare the posterior distributions of $K_{\eps}$ for the kinetic model \eqref{eq:chemotaxis} and the macroscopic model \eqref{eq:KellerSegel}, this choice is reasonable. As shown in \cite{chalub2004kinetic}, higher order terms in $\eps$ would not affect the macroscopic equation. Therefore they would not be reconstructable by the macroscopic inverse problem.
\end{remark}
In order to rigorously justify the above intuition on the convergence $f_{\eps}\to \rho F$ and ensure the existence of solutions to  equations \eqref{eq:chemotaxis}, \eqref{eq:KellerSegel}, we suppose $(K_0,K_1)$ to be an element of the admissible set 
\begin{align}\label{eq:admissibleset}
	\mathcal{A}= \{(K_0,K_1)\in	 &&\left(C^1(\mathbb{R}^3\times [0,\infty)\times V\times V)\right)^2\mid
	\|K_0\|_{C^1}, \|K_1\|_{C^1}\leq C  \text{ and }
	\\&&0<\alpha \leq K_0 \text{ symmetric }\text{ and }
	K_1 \text{ antisymmetric in }(v,v')\}\nonumber
\end{align}
for some preset constants $C,\alpha>0$. For any $(K_0,K_1)\in \mathcal{A}$ it is straightforward to show that
\begin{equation}\label{eqn:local_equi}
F\equiv 1/|V|\,,\quad\text{with}\quad |V| := \int_V 1\,dv\,.
\end{equation}

\begin{remark}
 With $(K_0,K_1)$ assumed to be symmetric and antisymmetric, the local equilibrium $F$ in \eqref{eqn:local_equi} is explicit and simple. This is e.g. the case for one typical choice of the tumbling kernel: $K[c,\nabla c]=a[c]+\eps b[c]\phi(v\cdot\nabla c-v'\cdot \nabla c)$ with antisymmetric $\phi$", which represents a special case of the models extensively studied in~\cite{chalub2004kinetic}. For better readability, we assume the symmetry properties of the tumbling kernel  stated in \eqref{eq:admissibleset} throughout the paper. We should mention, however, that it is possible to relax this assumptions on the tumbling kernel while maintaining the same macroscopic limit. In particular, if there exists one uniform velocity distribution $F(v)>0$  that is positive, bounded and satisfies
\begin{displaymath}
	\int_V F\, dv = 1,\hspace{0.5cm} \int_V vF(v) \, dv = 0 \hspace{0.5cm} \text{and} \hspace{0.5cm} 	K_0(x,t,v',v)F(v)= K_0(x,t,v,v')F(v')
\end{displaymath}
for all considered $K_0$ in the admissible set, then all statements and arguments provided in this paper still hold true.
Note that by these requirements, assumption (A0) in Chalub et al.\cite{chalub2004kinetic} is satisfied. 
\end{remark}

Suppose the initial data is smooth in the sense that $f_0 \in C^{1,+}_c(\rr^3\times V)$. ". Then we have the following theorem on convergence which can be viewed as an adaption of the results in \cite{chalub2004kinetic}.

\begin{theorem}\label{thm:fullConv}
Suppose $K_{\eps}$ has the form of~\eqref{eqn:K_form} with $(K_0,K_1)\in \mathcal{A}$ and suppose the initial condition $f_0\in C^{1,+}_c(\rr^3\times V)$, then the solution $f_\eps$ to the chemotaxis equation \eqref{eq:chemotaxis} satisfies the following:
\begin{enumerate}[label = \alph*)]
    \item \label{th:forwardExBound} For sufficiently small $\eps$, the solution $f_{\eps}$ of equation \eqref{eq:chemotaxis}  exists and is bounded in $L^{\infty}\big([0,T], L^1_+\cap L^{\infty}(\rr^3\times V)\big)$ for $T<\infty$.
    \item The solution $f_{\eps}$ converges to $\rho F $ in $L^{\infty}\big([0,T]; L^1_+\cap L^{\infty}(\rr^3\times V)\big)$, where $\rho$ satisfying the Keller-Segel equation \eqref{eq:KellerSegel} with coefficients
    \begin{eqnarray}
	D &=& \int_V v\otimes \kappa(x,t,v) \,dv \label{eq:DByK},\\
	\Gamma &=& -\int_V v\theta(x,t,v) \,dv\label{eq:GammaByK}.
    \end{eqnarray}
    Here $\theta$ and $\kappa$ solve the cell problems:
    \begin{eqnarray*}
    	\mathcal{K}_0(\kappa) = vF\,,\quad \text{and}\quad \mathcal{K}_0(\theta) = \mathcal{K}_1(F),
    \end{eqnarray*}
    where $\mathcal{K}_i(g):= \int_V K_{ {i}} g' - K_{ {i}}' g dv'$ for $i=0,1$.
    \item The boundedness and the convergence is uniform in $\mathcal{A}$.
\end{enumerate}
\end{theorem}

\begin{proof}[Sketch of proof]\hfill
\begin{enumerate}[label = \alph*)]
    \item \label{pr:boundf}
    First of all, we have the maximum principle so that
    \begin{eqnarray}
    \|f_{\eps}(\cdot,t,\cdot)\|_{L^1(\mathbb{R}^3\times V)} &= \|f_0\|_{L^1(\mathbb{R}^3\times V)} <\infty\,,
    \end{eqnarray}
    and following the same arguments as in~\cite{chalub2004kinetic}, we integrate in time for
    \begin{eqnarray}
        &f_{\eps}(x,t,v) &= f_0(x,v)+ \int_0^t \mathcal{K}_{\eps}(f_{\eps}) \left(x-\frac{vs}{\eps},t-s,v\right)\, ds \label{eq:fexplicitimplicitform}\\
        &&\leq f_0(x,v)+ \int_0^t \int_VK_{\eps}\left(x-\frac{vs}{\eps},t-s,v, v'\right)f_{\eps} \left(x-\frac{vs}{\eps},t-s,v'\right)\,dv'\, ds \nonumber\\
        &&\leq f_0(x,v)+ 2C\int_0^t \int_V f_{\eps} \left(x-\frac{vs}{\eps},t-s,v'\right)\,dv'\, ds\,.\nonumber
  	\end{eqnarray}
  	Noting that $f_0\in L^1_+\cap L^{\infty}$ and $0<K_{\eps} = K_0 + \eps K_1 \leq (1+\eps)C\leq 2C$ for  sufficiently small $\eps$, we have:
    \begin{eqnarray}\label{eq:estimate_fLinf}
  	\|f_{\eps}(\cdot,t,\cdot)\|_{L^{\infty}(\mathbb{R}^3\times V)}&  \leq 
  	\|f_{0}\|_{L^{\infty}(\mathbb{R}^3\times V)}  + 2C|V| \int_0^t \|f_{\eps}(\cdot,s,\cdot)\|_{L^{\infty}(\mathbb{R}^3\times V)}\, ds.
  	\end{eqnarray}
Calling the Gr\"onwall lemma one obtains a bound on $\|f_{\eps}(\cdot,t,\cdot)\|_{L^{\infty}(\mathbb{R}^3\times V)}$. Since the only role $K_i$ played is its boundedness by $C$, as in~\eqref{eq:fexplicitimplicitform}, the estimate we get is uniform in $\mathcal{A}$ and is independent of $\eps$ for $\eps$ small enough.
  	
  	\item 
We show that $f_{\eps}$ is a Cauchy sequence in $\eps$. For the purpose, we call $f_{\eps}$ and $f_{\tilde{\eps}}$ the solutions of the chemotaxis equation \eqref{eq:chemotaxis} with the scaling being $\eps$ and $\tilde{\eps}$. We also denote the difference $\hat{f}_{\eps, \tilde{\eps}}:= f_{\eps}-f_{\tilde{\eps}}$. Subtracting the two equations we have:
\begin{alignat}{4}
     &\eps^2 \partial_t\hat{f}_{\eps,\tilde{\eps}} + \eps v\cdot \nabla_x\hat{f}_{\eps,\tilde{\eps}} &=& \mathcal{K}_0(\hat{f}_{\eps, \tilde{\eps}}) + \eps \mathcal{K}_1(\hat{f}_{\eps,\tilde{\eps}}) \nonumber\\
     &&&  -(\eps^2-\tilde{\eps}^2)\partial_t f_{\tilde{\eps}} - (\eps-\tilde{\eps}) v\cdot \nabla_xf_{\tilde{\eps}} + (\eps-\tilde{\eps})\mathcal{K}_1(f_{\tilde{\eps}}) \label{eq:fCauchydiffeq}\\
    &&= & \mathcal{K}_{\eps} (\hat{f}_{\eps,\tilde{\eps}}) \underbrace{-(\eps^2-\tilde{\eps}^2)\partial_t f_{\tilde{\eps}} - (\eps-\tilde{\eps}) v\cdot \nabla_xf_{\tilde{\eps}} + (\eps-\tilde{\eps})\mathcal{K}_1(f_{\tilde{\eps}})}_{=:S}\nonumber
    \end{alignat}
with a trivial initial data $\hat{f}_{\eps,\tilde{\eps}}(x,0,v)=0$. This is an equation with a source term $S$. Using the argument as in~\ref{pr:boundf}, $L^{\infty}$ boundedness of the time and spatial derivative  $\partial_t f_{\tilde{\eps}},\,\nabla_x f_{\tilde{\eps}}$ in $S$ can be shown, meaning $S$ is of order $\eps-\tilde{\eps}$. Running~\eqref{eq:fexplicitimplicitform} again with this extra source term, we have
$$
\|f_{\eps}-f_{\tilde{\eps}}\|_{L^{\infty}([0,T];L^1\cap  L^{\infty}(\rr^3\times V))} =O(\eps -\tilde{\eps})\,.
$$ 
Hence $\{f_\eps\}$ is a Cauchy sequence, and thus converges to some $f\in L^{\infty}([0,T], L^1_+\cap L^{\infty}(\rr^3\times V))$. 


It remains to prove $f= \rho F$ almost everywhere in $[0,T]\times\rr^3\times V$ with $\rho$ satisfying the Keller-Segel equation \eqref{eq:KellerSegel} with $D$, $\Gamma$ as given in equations \eqref{eq:DByK}- \eqref{eq:GammaByK}. This follows by arguments rather similar to those in~\cite{chalub2004kinetic}, and is therefore omitted from here. Since only the boundedness of $(K_0,K_1)$ is seen in the proof, the convergence is uniform in $\mathcal{A}$.

    \end{enumerate}
\end{proof}

\section{Bayesian inverse problem setup}\label{sec:3setup}
Associated with the two forward models, there are two inverse problems. We describe the inverse problem setup and present them with the Bayesian inference formulation.

In the lab setup, it is assumed that the bacteria plate is large enough so that the boundary plays a negligible role. At the initial time, the bacteria cells are distributed on the plate. One then injects chemoattractants onto the plate through a controlled manner, so to have $c(t,x)$ explicitly given, forcing $K_{i}$, and $(D,\Gamma)$ to be functions of $(t,x,v)$ or $(t,x)$ only. The bacteria density at  location $x$ at time $t$ is then measured. 

Measuring is usually done by taking high resolution photos of the plate at time $t$ and counting the bacteria in a small neighbourhood of location $x$. Another possibility is  taking a sample of the bacteria at location $x$ and measuring the bacteria density of the sample by classical  techniques like optical density OD 600 or flow cytometry, see e.g.  \cite{ODMeasurement,FlowCytometry}. This however describes an invasive technique and thus  allows  measurements at only one  time $t$.

The whole experiment is to take data of the following  operator: 
\[
\mathcal{A}^\eps_{K_0\,,K_1}:\; f_0\to \int f_\eps(t,x,v)dv
\]
if the dynamics of the bacteria is modeled by~\eqref{eq:chemotaxis}, and
\[
\mathcal{A}^0_{K_0,K_1}=\mathcal{A}_{D\,,\Gamma}:\; \rho_0:= \int_V f_0dv \to \rho(t,x)
\]
if the dynamics of the bacteria is modeled by~\eqref{eq:KellerSegel}. Noting that $(D,\Gamma)$ are uniquely determined by $(K_0,K_1)$ by equations \eqref{eq:DByK},\eqref{eq:GammaByK}, we can equate $\mathcal{A}_{D\,,\Gamma}$ with $\mathcal{A}^0_{K_0,K_1}$.
Although the more natural macroscopic inverse problem would be to recover the diffusion and drift coefficients $D,\Gamma$  in \eqref{eq:KellerSegel}, we choose to formulate the inverse problem for the tumbling kernel $(K_0,K_1)$.
This allows us to compare the solution for both the kinetic and the macroscopic inverse problem.\\
\begin{remark} \label{rem:whichcompare}
In order to reasonably compare the solutions to the inverse problems, the solutions have to be of the same kind. 
We choose to reconstruct $(K_0,K_1)$ in both the kinetic and macroscopic inverse problem, see Figure \ref{fig:compareinvprob} (left). The macroscopic inverse problem is thus also formulated for $(K_0,K_1)$ which  $(D,\Gamma)$ is a function of. Alternatively one could also reconstruct $(D,\Gamma)$ from both models. In the kinetic setting this would mean to reconstruct $(K_0^{\chem},K_1^{\chem})$ and then transform to values of $(D^{\chem},\Gamma^{\chem})$ by equations \eqref{eq:DByK},\eqref{eq:GammaByK}, see Figure \ref{fig:compareinvprob} (right). \\
We do not choose this alternative, because the information on the tumbling kernel $(K_0,K_1)$ is microscopic and thus more detailed. Furthermore, with a fixed $(K_0,K_1)$, $(D,\Gamma)$ can be uniquely determined, and thus the convergence can be viewed as a mere consequence, see also Remark~\ref{rem:macroPosterior2ways}.
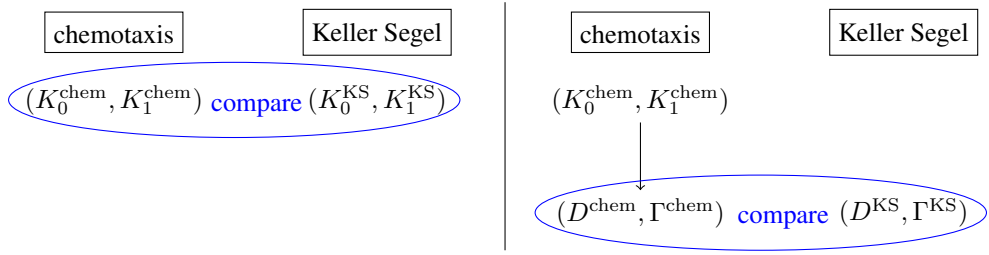
\begin{figure}[h]
   \begin{tikzpicture}
    \node [draw, text height=0.3cm] (1) at (0, 2.9) {chemotaxis};
    \node (2) at (0, 2) {$ (K_0^{\chem},K_1^{\chem})$};
 	\node [draw, text height=0.3cm] (4) at (3.5, 2.9) {Keller Segel};
	\node (5) at (3.5, 2) {$(K_0^{\KS},K_1^{\KS})$};
	\node (12) at (0, 1.7) {};
	\node (13) at (0, 0.3) {};
	\node (15) at (3.5, 1.7) {};
	\node (16) at (3.5, 0.3) {};
	\node (20) at (1.9, 1.9) {\textcolor{blue}{compare}};
	\draw [blue] (1.6,2) ellipse (3cm and 0.5cm); \vspace{3cm}
	
	\node (21) at (5.2,3.3){};
	\node (22) at (5.2,0){};
	\draw (21.center) to (22.center);
	
	\node [draw, text height=0.3cm] (1) at (7, 2.9) {chemotaxis};
    \node (2) at (7, 2) {$ (K_0^{\chem},K_1^{\chem})$};
	\node (3) at (7, 0.5) {$(D^{\chem},\Gamma^{\chem})$};
 	\node [draw, text height=0.3cm] (4) at (10.5, 2.9) {Keller Segel };
	\node (6) at (10.5, 0.5) {$(D^{\KS},\Gamma^{\KS})$};
	\node (12) at (7, 1.7) {};
	\node (13) at (7, 0.8) {};
	\node (15) at (10.5, 1.7) {};
	\node (16) at (10.5, 0.8) {};
	\node (20) at (8.9, 0.4) {\textcolor{blue}{compare}};
	\draw [->](12.center) to (13.center);
	\draw [blue] (8.6,0.5) ellipse (3cm and 0.5cm);
\end{tikzpicture}
\caption{Two ways to compare the inverse problems: determining and comparing the tumbling kernels for both underlying chemotaxis and Keller Segel models (left) or determining the drift or diffusion coefficient for the Keller Segel model and the tumbling kernel for the chemotaxis model and calculating the corresponding drift and diffusion coefficients.}
\label{fig:compareinvprob}
\end{figure}
\end{remark}
Multiple experiments can be conducted using different initial profile, but the same controlled $c(t,x)$ is used to ensure the to-be-reconstructed $K_{i}$ is unchanged from experiment to experiment. Denoting $k\in [1\,,\cdots,K]$ the indices of the different initial data setups, and $j=(j_1,j_2)\in[1\,,\cdots,J_1]\otimes[1\,,\cdots,J_2]$ the indices of the measuring time and location, with $t_j=t_{j_1}$ being the measuring time, and $\chi_{j}=\chi_{j_2}\in C_c(\mathbb{R}^3)$ being the spatial test function, then with~\eqref{eq:chemotaxis} and~\eqref{eq:KellerSegel} being the forward models, we take the measurements, respectively:
\begin{eqnarray}
	\mathcal{G}^{\eps, \chem}_{jk}(K_0,K_1) &=\mathcal{M}_j\left(\mathcal{A}^\eps_{K_0,K_1}(f_{0}^{(k)})\right)&= \int_{\mathbb{R}^3}\int_V f_{\eps}^{(k)}(x,t_j,v)\,dv \, \chi_j(x)dx\,, \label{eq:measureGchemo}\\
	\mathcal{G}^{\KS}_{jk}(K_0,K_1) &=\mathcal{M}_j\left(\mathcal{A}^0_{K_0,K_1}(\rho_0^{(k)})\right)&= \int_{\mathbb{R}^3}\rho^{(k)}(x,t_j)\, \chi_j(x)dx\,,\label{eq:measureGKS}
\end{eqnarray}
where $\mathcal{M}_j$ are the measuring operators with corresponding test functions $(\delta_j,\chi_j)$. One can think of $\chi_j$ a compactly supported blob function concentrated at a certain location, meaning all the bacteria cells in a small neighborhood are counted towards this particular measurement, see Figure~\ref{fig:demonstration_chi}. This is a reasonable model when counting bacteria in a small neighbourhood or taking samples with a pipette.
\begin{figure}[htb]
	\centering
	\includegraphics[width = 6cm]{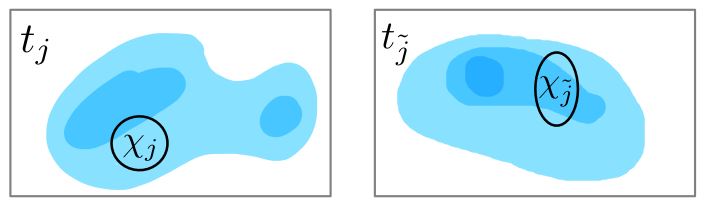}
	\caption{Measurement of the \textcolor{cyan}{bacteria density (blue)} at two different measuring times $t_j, t_{\tilde{j}}$. The location of the test functions is indicated by the support in space of the test functions $\chi_j, \chi_{\tilde{j}}$.}\label{fig:demonstration_chi}
\end{figure}\\
Throughout the paper we assume the initial data and the measuring operators are controlled:
\begin{equation}\label{eqn:bound_test}
\begin{aligned}
&\|f^{(k)}_0\|_{L^1}, \|f^{(k)}_0\|_{L^{\infty}}<C_\rho\,,\quad&&\forall k\\
&\max\{\|\chi_j\|_{L_1}\,,\|\chi_j\|_{L_2}\,,\|\chi_j\|_{L_\infty}, |\text{supp }\chi_j|_{dx}\}<C_x\,,\quad &&\forall j\,.
\end{aligned}
\end{equation}
\begin{remark}
The measurements $\mathcal{G}^{\eps, \chem}_{jk}(K_0,K_1), \mathcal{G}^{\KS}_{jk}(K_0,K_1)$ are formulated in a rather general form in equations \eqref{eq:measureGchemo},\eqref{eq:measureGKS} due to the freedom in the choice of the test function $\chi_j\in C_c(\rr^3)$.\\
However, all subsequent derivations also hold true for the specific case of pointwise measurements with $t_j:= t_{j_1}$ and $x_j := x_{j_2}$. The measurements would then be $\mathcal{G}^{\eps, \chem}_{jk}(K_0,K_1) = \int_V f_{\eps}^{(k)}(x_j,t_j,v)\,dv $ and $\mathcal{G}^{\KS}_{jk}(K_0,K_1) = \rho^{(k)}(x_j,t_j)$, which would correspond to measuring operators $\mathcal{M}_j$ with test functions $(\delta_{t_{j_1}},\delta_{x_{j_2}})$.
\end{remark}

Since measuring error is not avoidable in the measuring process, we assume it introduces additive error and collect the data of the form
\begin{eqnarray*}
	y_{jk}^{\eps, \chem}	&=&\mathcal{G}^{\eps, \chem}_{jk}(K_0,K_1) +\eta_{jk},\\
	y_{jk}^{\KS}&	=&\mathcal{G}^{\KS}_{jk}(K_0,K_1) +\eta_{jk}\,,
\end{eqnarray*}
where the noise $\eta_{jk}$ is assumed to be a random variable independently drawn from a Gaussian distribution $N(0,\gamma^2)$ of known variance $\gamma^2 >0$.

In the Bayesian form, the to-be-reconstructed parameter $(K_0,K_1)$ is assumed to be a random variable, and the goal is to reconstruct its distribution. Suppose a-priori we know that the parameter is drawn from the distribution $\mu_0$, then the Bayesian posterior distributions for $(K_0,K_1)$ should be
\begin{equation}\label{eq:PostChemo}
\begin{aligned}
	\mu^{y}_{\eps, \chem}(K_0,K_1) &= \frac{1}{Z^{\eps,\chem}}\mu^{(K_0,K_1)}_{\eps,\chem}(y)\,\mu_0(K_0,K_1)\\
	&= \frac{1}{Z^{\eps,\chem}}e^{-\frac{1}{2\gamma^2} \|\mathcal{G}^{\eps,\chem}(K_0,K_1) -y \|^2}\,\mu_0(K_0,K_1)\,,
	\end{aligned}
\end{equation}
using~\eqref{eq:chemotaxis} as the forward model, and
\begin{equation}\label{eq:PostKS}
\begin{aligned}
	\mu^{y}_{\KS}(K_0,K_1)& = \frac{1}{Z^{\KS}}\mu^{(K_0,K_1)}_{\KS}(y)\,\mu_0(K_0,K_1)\\
	&= \frac{1}{Z^{\KS}}e^{-\frac{1}{2\gamma^2} \|\mathcal{G}^{\KS}(K_0,K_1) -y \|^2}\,\mu_0(K_0,K_1)\,,
	\end{aligned}
\end{equation}
using~\eqref{eq:KellerSegel} as the forward model. In the formula $Z^{\circ}$ is the normalization constant to ensure $\int 1d\mu^y_\circ(K_0,K_1)=1$ and 
\begin{displaymath}
\mu^{(K_0,K_1)}_{\circ}(y)=e^{-\frac{1}{2\gamma^2} \|\mathcal{G}^{\circ}(K_0,K_1) -y \|^2}
\end{displaymath}
is the likelihood of observing the data $y$ from a model with a tumbling kernel or diffusion and drift term derived by $(K_0,K_1)$.\\
In Section~\ref{sec:4PosteriorConv} we need to specify the conditions on $\mu_0$ to ensure the well-definedness of $\mu^y_\circ$.\\
\begin{remark}\label{rem:macroPosterior2ways}
Since the macroscopic model does not explicitly depend on $(K_0,K_1)$, the distribution of $\mu_{\KS}^y(D,\Gamma)$ is of interest in the macroscopic description \eqref{eq:KellerSegel}. There are two ways to derive it starting with a prior distribution on $(K_0,K_1)$: The natural way would be to transform the prior distribution to a prior on $(D,\Gamma)$ by equations \eqref{eq:DByK}-\eqref{eq:GammaByK} and then consider the inverse problem of reconstructing $(D,\Gamma)$. This approach is displayed by the lower path in Figure \ref{fig:macroPosterior2ways}. If, however, the posterior distribution $\mu_{\KS}^y(K_0,K_1)$ is calculated ahead of the transformation (as in our case), one could instead  transform this posterior distribution directly to a distribution in the $(D,\Gamma)$ space following the upper path in Figure \ref{fig:macroPosterior2ways}. Naturally the question arises whether the two ways lead to the same posterior distribution. It turns out they do. Considering the second possibility, we see that the likelihood and thus the normalization constant only depend on $(D,\Gamma)$, because we are in the macroscopic model. Hence, only the prior distribution is transformed just like it is the case for the first possibility.
\begin{figure}[H]
   \begin{tikzpicture}
    \node (2) at (0, 2) {$\mu_0(K_0,K_1)$};
	\node (3) at (0, 0) {$\mu_0(D,\Gamma)$};
	\node (5) at (5.5, 2) {$\mu_{\KS}^y(K_0,K_1)$};
 	\node (6) at (5.5, 0) {$\mu_{\KS}^y(D,\Gamma)$};
	\node (12) at (0, 1.7) {};
	\node (13) at (0, 0.3) {};
	\node (14) at (0.9,1){transform};
	\node (15) at (5.5, 1.7) {};
	\node (16) at (5.5, 0.3) {};
	\node (23) at (6.4,1){transform};
	\node (17) at (1,2){};
	\node (18) at (4.5,2){};
	\node (19) at (2.75,2.2){inverse problem};
	\node (20) at (1,0){};
	\node (21) at (4.5,0){};
	\node (22) at (2.75,0.2){inverse problem};
	\draw [->](12.center) to (13.center);
	\draw [->](17.center) to (18.center);
 	\draw [->] (15.center) to (16.center);
 	\draw [->] (20.center) to (21.center);
\end{tikzpicture}
\caption{Two ways to determine the posterior distribution $\mu_{\KS}^y(D,\Gamma)$ from a prior $\mu_0(K_0,K_1)$ on the tumbling kernels.}
\label{fig:macroPosterior2ways}
\end{figure}
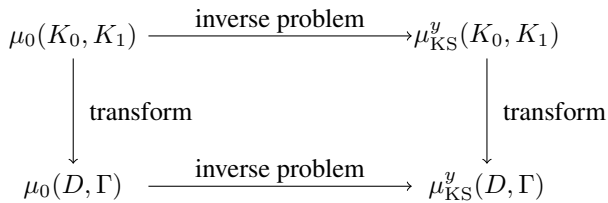
\end{remark}

\section{Convergence of posterior distributions}\label{sec:4PosteriorConv}
One natural question arises: the two different forward models provide two different posterior distribution functions of $(K_0,K_1)$. Which distribution is the correct one, or rather, what is the relation between the two posterior distributions?

As discussed in section~\ref{sec:2setup}, the two forward models are asymptotically equivalent in the long time large space regime, so it is expected that the two posterior distribution converge as well. This suggests the amount of information given by the measurements is equally presented by the two forward models. However, this convergence result is not as straightforward as it may seem. One issue comes from the control of initial data and the measurement operator. For each initial data, the solution converges in $L^{\infty}\big([0,T]; L^1_+\cap L^{\infty}(\rr^3\times V)\big)$, we now have a list of initial data, and the solutions are tested on a set of measuring operators, so we need a uniform convergence when tested on the dual space. Furthermore, to show the convergence of two distribution functions, a certain metric needs to be given on the probability function space, how does the convergence for one set of fixed $(K_0,K_1)$ translates to the convergence on the entire admissible set also needs to be taken care of.

By choosing the admissible set \eqref{eq:admissibleset}, we formulated an assumption on the tumbling kernels $(K_0,K_1)$ ahead of time. With this a priori knowledge we showed the uniform boundedness and convergence of the solutions $f_{\eps}$ to the chemotaxis equation \eqref{eq:chemotaxis} over the function set $\mathcal{A}$ in Theorem \ref{thm:fullConv}. This will play a crucial role in the convergence proof for the inverse problem. From here and on, we assume the prior distribution $\mu_0$ is supported on $\mathcal{A}$.

Before diving in to show the convergence, as an a priori estimate, we first show the well-posedness of the Bayesian posterior distributions in Lemma~\ref{lem:welldefPosteriormeasures}, following~\cite{stuart2010inverse, Stuart2015BayesianInversionHandbookofUQ}.
\begin{lemma}\label{lem:welldefPosteriormeasures}
If the initial conditions $f_0^{(k)}\in C^{1,+}_c(\rr^3\times V)$ and the test functions $\chi_j\in C_c(\rr^3)$ satisfy \eqref{eqn:bound_test} then the following  properties of the posterior distributions hold true:
\begin{enumerate}[label = \alph*)]
    \item The measurements $\mathcal{G}^{\eps,\chem}$ and $\mathcal{G}^{\KS} $ are uniformly bounded  on $\mathcal{A}$ (and uniformly in $\eps$).
    \item For  small enough $\eps$, the measurements $\mathcal{G}^{\eps,\chem}$ and $\mathcal{G}^{\KS} $ are Lipschitz continuous with respect to the tumbling kernels $(K_0,K_1)$ under the norm $\|(K_0,K_1)\|_{*}:= \max(\|K_0\|_{\infty},\|K_1\|_{\infty} )$ on $\mathcal{A}$.
    \item 	The posterior distributions are well-posed and absolutely continuous w.r.t. each other.
\end{enumerate}
\end{lemma}

\begin{proof}
	\begin{enumerate}[label = \alph*)]
		\item For every $(j,k)$, we have:
		\begin{eqnarray*}
			|\mathcal{G}^{\KS}_{jk}(K_0,K_1)| &=& \left|\int_{\mathbb{R}^3}\rho^{(k)}(x,t_j)\, \chi_j(x)dx\right|\\
			&\leq &\|\chi_j(x)\|_{\infty} \|\rho^{(k)}(\cdot, t_j)\|_{L^1(\mathbb{R}^3)}= \|\chi_j(x)\|_{\infty} \|\rho_0^{(k)}\|_{L^1(\mathbb{R}^3)} \\
			&\leq& C_{x} C_{\rho}
		\end{eqnarray*}
		where we used the density conservation: $\|\rho(\cdot,t)\|_{L^1(\mathbb{R}^3)}= \|\rho_0\|_{L^1(\mathbb{R}^3)}$ for all $t$. Analogously we have $|\mathcal{G}^{\eps, \chem}_{jk}(K_0,K_1)|\leq C_x C_{\rho}$. Note that this bound is independent of $\eps$.
		\item 
		For the chemotaxis model, we have for $(K_0,K_1),(\tilde{K_0},\tilde{K_1})\in \mathcal{A}$
		\begin{align}
			|&\mathcal{G}^{\eps,\chem}_{jk}(K_0,K_1) - \mathcal{G}^{\eps, \chem}_{jk}(\tilde{K_0},\tilde{K_1})|  =\left|\int_{\rr^3}\int_V (f ^{(k)}_{\eps}-\tilde{f} ^{(k)}_{\eps}) (x,t_j,v)\, dv \, \chi_j(x)\, dx\right|\nonumber\\
			& \leq \|\chi_j\|_{\infty}\int_{\text{supp }\chi_j} \int_V |\bar{f} ^{(k)}_{\eps}(x,t_j,v)|\,dv\, dx 
			 \leq C_x|V||\text{supp }\chi_j|_{dx} \|\bar{f} ^{(k)}_{\eps}(\cdot,t_j,\cdot)\|_{L^{\infty}(\rr^3\times V)}\nonumber\\
			 &\leq C_x^2|V| \|\bar{f} ^{(k)}_{\eps}(\cdot,t_j,\cdot)\|_{L^{\infty}(\rr^3\times V)} \label{eq:estContGchem},
		\end{align}
		where $f ^{(k)}_{\eps}$  and $\tilde{f} ^{(k)}_{\eps}$ are  solutions to the initial value problem (\ref{eq:chemotaxis}) with initial condition $f ^{(k)}_0$ and tumbling kernels $K_{\eps} = K_0+\eps K_1$ and $\tilde{K}_{\eps} = \tilde{K}_0+\eps \tilde{K}_1$ respectively. Their difference $\bar{f} ^{(k)}_{\eps}:= f ^{(k)}_{\eps}-\tilde{f} ^{(k)}_{\eps}$ satisfies the scaled  difference  equation:
		\begin{align*}
			{\eps^2}\frac{\partial }{\partial t}\bar{f} ^{(k)}_{\eps}(x,t,v) +  {\eps} v\cdot \nabla_x \bar{f}^{(k)}_{\eps}(x,t,v) &= \tilde{\mathcal{K}}_{ {\eps}}(\bar{f} ^{(k)}_{ {\eps}}) + \bar{\mathcal{K}}_{ {\eps}}(f ^{(k)}_{ {\eps}})\\
			\bar{f} ^{(k)}_{\eps}(x,0,v)&= 0.
		\end{align*}
		Here,  $\bar{\mathcal{K}}$ denotes the tumbling operator with kernel $\bar{K}_{\eps}:=K_{\eps}-\tilde{K}_{\eps}$. Integration in $s$ at $(x-\frac{vs}{\eps},t-s,v)$ shows 
		\begin{align*}
			\bar{f}_{\eps}^{(k)}(x,t,v)&= \int_0^t &&\tilde{\mathcal{K}}_{ {\eps}}(\bar{f}^{(k)}_{ {\eps}}) \left(x-\frac{vs}{\eps},v,t-s\right) + \bar{\mathcal{K}}_{ {\eps}}(f^{(k)}_{ {\eps}})\left(x-\frac{vs}{\eps},v,t-s\right) \, ds\\
			&=\int_0^t &&\int_V\tilde{{K}}_{ {\eps}}\bar{f}'^{(k)}_{ {\eps}} \left(x-\frac{vs}{\eps},v,v',t-s\right) -\tilde{{K}}'_{ {\eps}}\bar{f}^{(k)}_{ {\eps}} \left(x-\frac{vs}{\eps},v,v',t-s\right) \, dv'\\
			&&& + \int_V\bar{{K}}_{ {\eps}}f'^{(k)}_{ {\eps}} \left(x-\frac{vs}{\eps},v,v',t-s\right) -\bar{{K}}'_{ {\eps}}f^{(k)}_{ {\eps}} \left(x-\frac{vs}{\eps},v,v',t-s\right) \, dv'\,ds.
		\end{align*}
		This yields
		\begin{align*}
			\|\bar{f}_{\eps} ^{(k)}(\cdot,t,\cdot)\|_{L^{\infty}(\rr^3\times V)} \leq &2 \|\tilde{K}_{\eps}\|_{\infty} |V|\int_0^t\|\bar{f}_{\eps} ^{(k)}(\cdot, t-s, \cdot)\|_{L^{\infty}(\rr^3\times V)} \, ds \\
			&+ 2 \|K_{\eps}-\tilde{K_{\eps}}\|_{\infty} |V| \|f_{\eps} ^{(k)}\|_{\infty} t\\
			\leq & 4 C |V|\int_0^t\|\bar{f}_{\eps} ^{(k)}(\cdot, s, \cdot)\|_{L^{\infty}(\rr^3\times V)} \, ds \\
			&+ 4 \|(K_0-\tilde{K}_0, K_1-\tilde{K}_1 )\|_{*} |V| c_f T
		\end{align*}
		since one has $\|K_{\eps}\|_{\infty} \leq 2 \|(K_{0},K_1)\|_{*} \leq 2C $ for small enough $\eps< 1$ and $f_{\eps} ^{(k)}\leq c_f $ 
		is bounded in $L^{\infty}$ uniformly on $\mathcal{A}$ by Theorem \ref{thm:fullConv} \ref{th:forwardExBound}. Additionally, $c_f$ can be chosen to be independent of $k$ by inserting the uniform boundedness of $\|f_0^{(k)}\|_{L^{\infty}}$ in \eqref{eqn:bound_test} into equation \eqref{eq:estimate_fLinf}. The Gr\"onwall Lemma thus gives 
		\begin{displaymath}
				\|\bar{f}_{\eps} ^{(k)}(\cdot,t,\cdot)\|_{L^{\infty}(\rr^3\times V)} \leq  L(T,C, C_{\rho})\|(K_0-\tilde{K}_0, K_1-\tilde{K}_1 )\|_{*}
		\end{displaymath}
		with some coefficient $L$ depending on  $T$, $C$ and $C_{\rho}$. Inserting this in equation~\eqref{eq:estContGchem} 
		results in the desired Lipschitz continuity.
		
		We similarly study the Lipschitz continuity of the Keller-Segel measurements $\mathcal{G}^{\KS}_{jk}(K_0,K_1)$. The proof strategy is almost the same. With some computational effort, one can see:
		\begin{eqnarray*}
		        |\mathcal{G}^{\KS}_{jk}(K_0,K_1) &-& \mathcal{G}^{\KS}_{jk}(\tilde{K}_0,\tilde{K}_1)| \leq \|\chi_j\|_{L^2} \|(\rho^{(k)}-\tilde{\rho} ^{(k)})(\cdot, t_j)\|_{L^2}\\
		        &\leq&  C_x c(\|D-\tilde{D}\|_{L^{\infty}([0,T]\times \rr^3; \rr^{3\times3})}+\|\Gamma-\tilde{\Gamma}\|_{L^{\infty}([0,T]\times \rr^3; \rr^{3})}) 
		\end{eqnarray*}
		where $(\Gamma, D), (\tilde{\Gamma},\tilde{D})$ are the drift and diffusion terms derived by the collision operators defined by $(K_0,K_1)$ and $(\tilde{K}_0,\tilde{K}_1)$ respectively by equations \eqref{eq:DByK}-\eqref{eq:GammaByK}. The constant $c$ monotonously depends  on  the $L^2$ norms of $\rho^{(k)}$ and $\nabla_x \rho^{(k)}$ which are bounded uniformly on $\mathcal{A}$. By the linear relation between $D$ and $\kappa$ and $\Gamma$ and $\theta$, this directly translates to 
		\begin{eqnarray*}
		 |\mathcal{G}^{\KS}_{jk}(K_0,K_1) - \mathcal{G}^{\KS}_{jk}(\tilde{K}_0,\tilde{K}_1)| \leq \tilde{c}cC_x \big(&\|\kappa - \tilde{\kappa}\|_{L^{\infty}([0,T]\times\rr^3;L^2(V; \frac{dv}{F}; \rr^{3}))} \\
		 &+ \|\theta - \tilde{\theta}\|_{L^{\infty}([0,T]\times\rr^3;L^2(V; \frac{dv}{F}))}\big),
		\end{eqnarray*}
		with constant $\tilde{c}$ depending only on $V$. Finally, the Lax-Milgram theorem shows the continuous dependence of
		\begin{displaymath}
		       \|\theta - \tilde{\theta}\|_{L^2(V; \frac{dv}{F})}+\|\kappa - \tilde{\kappa}\|_{L^2(V; \frac{dv}{F}; \rr^{3})}\leq \hat{c} \|(K_0,K_1)- (\tilde{K}_0,\tilde{K}_1)\|_{*}
		\end{displaymath}
		 where $\hat{c}$ only depends on $V,\alpha, C$.
		\item By a), the likelihoods $e^{-\frac{1}{2\gamma^2} \|\mathcal{G}^{\circ}(K_0,K_1) -y \|^2}$ are bounded away from zero and bounded  uniformly on $\mathcal{A}$ (and  in $\eps$). Thus, also the normalization constants $Z$ are. Part b) guarantees the measurability of the likelihoods. In total, this shows that the posterior distributions are well-defined and continuous with respect to each other. 	Since the likelihoods are continuous in $y$, well-posedness of the posterior distributions is given.
	\end{enumerate}
\end{proof}

We are now ready to show the convergence of the two posterior measures. There are two quantities we use to measure the difference between two distributions:
\begin{itemize}
\item Kullback-Leibler divergence
\[
d_{\KL}(\mu_1,\mu_2):=\int_{\mathcal{A}}\left(\log\frac{d\mu_1}{d \mu_2}(u)\right)d \mu_2(u)
\]
\item Hellinger metric
\[
d_{\Hell}(\mu_{1}, \mu_{2})^2 = \frac{1}{2}\int_{\mathcal{A}}\left(\sqrt{\frac{d \mu_{1}}{d\mu_{0}}(u)} - \sqrt{\frac{d \mu_{2}}{d\mu_{0}}(u)} \right)^2d\mu_{0}(u).
\]
\end{itemize}
The two metrics both evaluate the distance between the two probability measures $\mu_1$ and $\mu_2$ that are either absolutely continuous with respect to each other or with respect to a third probability measure $\mu_0$. Both are frequently used for comparing two distribution functions e.g. in Machine Learning \cite{KLinML1,KLinML2,KLinML3,HellinML1,HellinML2,HellinML3} or inverse problem settings \cite{newton2020diffusive,Abdulle2018BayesianConvParametPDE}. Even though the Kullback-Leibler divergence lacks the symmetry and triangle-inequality properties of a metric, it gained popularity due to its close connection to several information concepts such as the Shannon entropy or the Fisher information metric  \cite{KLdivergence}. Conversely, the Hellinger metric is a true metric. Although it does not have a demonstrative interpretation as the Kullback-Leibler divergence, its strength lies in the fact that convergence in the Hellinger metric implies convergence of the expectation of any polynomially bounded function with respect to either of the posterior distributions, as explained in \cite{stuart2010inverse}. In particular the mean, covariance and further moments of the distributions converge.

Before comparing the posterior measures, we need to have a look at the convergence of the measurements $\mathcal{G}^{\circ}(K_0,K_1)$.
\begin{lemma}\label{lem:measurementopconv}
	Assuming the initial and testing functions satisfy~\eqref{eqn:bound_test}, the chemotaxis measurements $\mathcal{G}^{\eps, \chem}$ converge to the Keller-Segel measurements $\mathcal{G}^{\KS}$ uniformly on $\mathcal{A}$ as $\eps\to 0$.
\end{lemma}
\begin{proof}
Theorem \ref{thm:fullConv} shows the convergence of $f_{\eps}$ to $\rho F$ in $L^{\infty}([0,T],L^1_+\cap L^{\infty}(\rr^3\times V))$ uniformly on $\mathcal{A}$. As a consequence, we have the convergence of the measurements:
	\begin{eqnarray*}
		&& \left|\mathcal{G}^{\eps, \chem}_{jk}(K_0,K_1) - \mathcal{G}^{\KS}_{jk}(K_0,K_1)\right|\\
		&=& \left|\int_{\mathbb{R}^3}\int_V f_{\eps}^{(k)}(x,t_j,v)\,dv \, \chi_j(x)dx -	\int_{\mathbb{R}^3} \,\rho^{(k)}(x,t_j) \chi_j(x)dx\right| \\
		&\leq& \int_{\mathbb{R}^3}\int_V |f_{\eps}^{(k)}(x,t_j,v) -\rho^{(k)}(x,t_j)F(v)|\,dv\, |\chi_j(x)|dx \\
		&\leq& \|f_{\eps}^{(k)}(\cdot,t_j,\cdot) -\rho^{(k)}(\cdot,t_j)F\|_{L^{\infty}(\rr^3\times V)}|V| \|\chi_j\|_{L^1(\rr^3)}\\
		&\to& 0
	\end{eqnarray*}
	where we used the form $F=\frac{1}{V}$. By the uniform convergence of $f_\eps$ to $\rho F$, this holds uniformly on $\mathcal{A}$. Since initial data and measuring test functions that satisfy~\eqref{eqn:bound_test} we have the uniform convergence over $(j,k)$ as well.
\end{proof}
We can now proof the following theorem on the asymptotic equivalence of the two posterior measures describing the distribution of the tumbling kernels $(K_0,K_1)\in \mathcal{A}$ if the dynamics of the bacteria is modelled by the kinetic \eqref{eq:chemotaxis} or macroscopic equation \eqref{eq:KellerSegel}.
\begin{theorem}\label{th:inverse}
Let the measurement  of the macroscopic bacteria density  be of the form  \eqref{eq:measureGchemo} and \eqref{eq:measureGKS} for a underlying kinetic chemotaxis model or a Keller Segel model respectively. The measuring test functions $\chi_j\in C_c(\rr^3)$ and initial data $f_0^{(k)} \in C_c^{1,+}(\rr^3\times V)$ are assumed to satisfying \eqref{eqn:bound_test}. Given a prior distribution $\mu_0$ on $\mathcal{A}$ and an additive centered Gaussian noise in the data, the posterior distribution for the tumbling kernel derived from the kinetic chemotaxis equation and the macroscopic Keller Segel equation as underlying models are asymptotically equivalent in the Kullback Leibler divergence 
	\begin{alignat*}{3}
		d_{\KL}(\mu^{y}_{{\eps}, \chem},	\mu^{y}_{\KS})
		\xrightarrow{{\eps}\to 0} 0.
	\end{alignat*}
\end{theorem}
\begin{proof}[Proof of Theorem \ref{th:inverse}]
	With the above Lemmas one can proceed as in the proof in \cite{newton2020diffusive}.
	The integrand of the Kullback-Leibler divergence is by the definition of the normalization constants of order
	\begin{eqnarray*}
		\log\frac{d\mu^{y}_{\eps, \chem}}{d \mu^{y}_{\KS}}(K_0,K_1) &=& \log \left(\frac{\mu_0(K_0,K_1)\mu^{(K_0,K_1)}_{\eps, \chem}(y)}{Z^{\eps,\chem}}\frac{Z^{\KS}}{\mu_0(K_0,K_1)\mu^{(K_0,K_1)}_{\KS}(y)}\right)\\
		&=&\log \frac{Z^{\KS}}{Z^{\eps, \chem}} + \log \frac{\mu^{(K_0,K_1)}_{\eps,\chem}(y)}{\mu^{(K_0,K_1)}_{\KS}(y)} \\
		&=& \mathcal{O}(|Z^{{\eps},\chem}-Z^{\KS}|) + \mathcal{O}(|\mu^{(K_0,K_1)}_{\eps,\chem}(y)-\mu^{(K_0,K_1)}_{\KS}(y)|) \\
		&=& \mathcal{O}(|\mu^{(K_0,K_1)}_{\eps,\chem}(y)-\mu^{(K_0,K_1)}_{\KS}(y)|).
	\end{eqnarray*}\vspace{-0.3cm}
	Thus, we estimate 
	\begin{eqnarray*}
		&&|\mu^{(K_0,K_1)}_{\eps,\chem}(y) - \mu^{(K_0,K_1)}_{\KS}(y)|\\
		&&= \left|\exp\left( -\frac{\|y-\mathcal{G}^{\eps,\chem}(K_0,K_1)\|^2}{2\gamma^2}\right)-\exp\left( -\frac{\|y-\mathcal{G}^{\KS}(K_0,K_1)\|^2}{2\gamma^2}\right)\right|\\
		&\leq& c \left|\|y-\mathcal{G}^{\eps, \chem}(K_0,K_1)\|^2 - \|y-\mathcal{G}^{\KS}(K_0,K_1)\|^2\right|
	\end{eqnarray*}
	for the Lipschitz constant $c < \infty$ of $\exp(-\frac{|x|}{2\gamma^2})$
	and
	\begin{flalign*}
		&\left|\|y-\mathcal{G}^{\eps, \chem}(K_0,K_1)\|^2 - \|y-\mathcal{G}^{\KS}(K_0,K_1)\|^2\right|\\
		&= \left|\text{tr}\left[\left(2y-\mathcal{G}^{\eps, \chem}(K_0,K_1)-\mathcal{G}^{\KS}(K_0,K_1)\right)^T \left(\mathcal{G}^{\eps, \chem}(K_0,K_1)-\mathcal{G}^{\KS}(K_0,K_1)\right)\right]\right|\\
		&\leq \|2y-\mathcal{G}^{\eps, \chem}(K_0,K_1)-\mathcal{G}^{\KS}(K_0,K_1)\|\cdot\|\mathcal{G}^{\eps, \chem}(K_0,K_1)-\mathcal{G}^{\KS}(K_0,K_1)\|.
	\end{flalign*}
	The first factor is bounded uniformly on $\mathcal{A}$ and in $\eps$ by Lemma \ref{lem:welldefPosteriormeasures} a) and  Lemma \ref{lem:measurementopconv} shows that  the second factor converges to $0$  uniformly on $\mathcal{A}$.
	It follows that
	\begin{displaymath}
		d_{\KL}(\mu^{y}_{{\eps}, \chem},	\mu^{y}_{\KS}) \to 0.
	\end{displaymath}
\end{proof}
The boundedness of the  the Hellinger metric by the Kullback Leibler divergence 
\begin{equation*}
    	d_{\Hell}^2(\mu_1,	\mu_2) \leq d_{\KL}(\mu_1,	\mu_2) 
\end{equation*}
 as shown in Lemma 2.4 in \cite{KLHellestimate} together theorem \ref{th:inverse} yield the asymptotic equivalence of the posterior distributions also in the Hellinger metric.
\begin{corollary}
    In the  framework of Theorem \ref{th:inverse}, one has 
	\begin{alignat*}{3}
		d_{\Hell}(\mu^{y}_{{\eps}, \chem},	\mu^{y}_{\KS})
		\xrightarrow{{\eps}\to 0} 0.
	\end{alignat*}
\end{corollary}

\section{Summary and Discussion}\label{sec:5Summary}

In this article, we considered bacterial movement in an environment with an attracting chemical substance that was not produced or consumed by the bacteria. The bacteria density was modelled to follow a chemotaxis equation~\eqref{eq:chemotaxis} on the kinetic level and a Keller-Segel equation~\eqref{eq:KellerSegel} on the macroscopic level. We studied the reconstruction of the tumbling coefficient using the measurement of the bacteria density at different time and location using different initial data. After adapting the results from~\cite{chalub2004kinetic} in the parabolic scaling, we study the equivalence between the reconstructions using the two different underlying models in the Bayesian framework. 
Assumptions on the prior information were made to guarantee the uniform convergence of the two forward models. This enabled us to show  that the posterior distributions are properly defined and that  convergence of the two posterior distributions holds true. The distance between two posterior distributions was measured in both the Kullback-Leibler divergence and the Hellinger metric.

The work presented here serves as a cornerstone of future research. On one hand, the study here can help design an efficient inversion solver. Most inversion solvers are composed of many iterations of forward solvers. Since kinetic chemotaxis equation lies on the phase space and is numerically much more expensive, the limiting Keller-Segel equation can serve as a good substitute for generating a good initial guess and speeding up the computation. On the other hand, the approach performed in this study is rather general, and with small modification, it also provides the foundation for explaining experiments, such as~\cite{giometto2015generalized}.




\vspace{6pt} 



\subsection*{Funding}
K.H. acknowledges support by the \emph{W\"urzburg Mathematics Center for Communication and Interaction} (WMCCI) as well as the \emph{Studienstiftung des deutschen Volkes} and the \emph{Marianne-Plehn-Programm}.\\
Q.L. acknowledges support from Vilas Early Career award. The research is supported in part by NSF via grant DMS-1750488 and Office of the Vice Chancellor for Research and Graduate Education at the University of Wisconsin Madison with funding from the Wisconsin Alumni Research Foundation.\\
M.T. acknowledge the support by NSFC11871340 and Changjiang Scholar Program-Youth Project.
%

\bibliographystyle{unsrtnat}
\bibliography{lit.bib}

\end{document}